\documentclass[11pt,A4paper]{article}

\usepackage[left=32mm,right=32mm,top=30mm,bottom=32mm]{geometry}
\usepackage{epsfig}
\usepackage{epstopdf}

\usepackage{caption}

\usepackage{xfrac}

\usepackage[percent]{overpic}

\usepackage{tikz}
\usetikzlibrary{positioning}
\usetikzlibrary{arrows}

\usepackage{color}
\usepackage{amsthm,amsmath,amssymb}
\usepackage{booktabs}
\usepackage{mathpazo}
\usepackage{microtype}
\usepackage{overpic}
\usepackage{bm}
\usepackage{sectsty}
\usepackage[
	pdftitle={PDFTitle},
	pdfauthor={Hugo Parlier},
	ocgcolorlinks,
	linkcolor=linkred,
	citecolor=linkred,
	urlcolor=linkblue]
{hyperref}

\usepackage{mathtools}

\definecolor{linkred}{RGB}{199,21,133}
\definecolor{linkblue}{RGB}{16, 78, 139}

\usepackage[hang,flushmargin]{footmisc}
\usepackage{enumitem}
\usepackage{titlesec}
	\titlespacing{\section}{0pt}{12pt}{0pt}
	\titlespacing{\subsection}{0pt}{6pt}{0pt}
	
\titlelabel{\thetitle.\quad}

\makeatletter 

\long\def\@footnotetext#1{% 
\H@@footnotetext{% 
\ifHy@nesting 
\hyper@@anchor{\@currentHref}{#1}% 
\else 
\Hy@raisedlink{\hyper@@anchor{\@currentHref}{\relax}}#1% 
\fi 
}}

\def\@footnotemark{% 
\leavevmode 
\ifhmode\edef\@x@sf{\the\spacefactor}\nobreak\fi 
\H@refstepcounter{Hfootnote}% 
\hyper@makecurrent{Hfootnote}% 
\hyper@linkstart{link}{\@currentHref}% 
\@makefnmark 
\hyper@linkend 
\ifhmode\spacefactor\@x@sf\fi 
\relax 
}% 

\ifFN@multiplefootnote% 
\renewcommand*\@footnotemark{% 
\leavevmode 
\ifhmode 
\edef\@x@sf{\the\spacefactor}% 
\FN@mf@check 
\nobreak 
\fi 
\H@refstepcounter{Hfootnote}% 
\hyper@makecurrent{Hfootnote}% 
\hyper@linkstart{link}{\@currentHref}% 
\@makefnmark 
\hyper@linkend 
\ifFN@pp@towrite 
\FN@pp@writetemp 
\FN@pp@towritefalse 
\fi 
\FN@mf@prepare 
\ifhmode\spacefactor\@x@sf\fi 
\relax% 
}% 
\fi 

\makeatother 

\newtheorem{thm}{Theorem}[section]

\newtheorem{coro}[thm]{Corollary}
\newtheorem{lem}[thm]{Lemma}

\theoremstyle{definition}

\theoremstyle{remark}

\newtheorem{remk}[thm]{Remark}

\renewcommand{\phi}{\varphi}

\newcommand{\tmu}{\widetilde{\mu}}

\newcommand{\omu}{\overline{\mu}}

\newcommand{\Fcal}{ {\mathcal F}}

\newcommand{\Acall}{ {\mathcal A}}

\newcommand{\be}{ \begin{equation} }
\newcommand{\ee}{ \end{equation} }
\newcommand{\hra}{\hookrightarrow}
\newcommand{\ra}{\rightarrow}
\newcommand{\lra}{\longrightarrow}

\newcommand{\inv}{^{-1}}

\newcommand{\Aut}{\mathrm{Aut}}
\newcommand{\Mod}{\mathrm{Mod}}
\newcommand{\St}{\mathrm{St}}

\newcommand{\arcgraph}[2]{\mathcal{A}^{[#1]}(#2)}
\newcommand{\arcgraphstratum}[3]{\mathcal{A}^{[#1]}_{#2}(#3)}

\newcommand{\B}[2]{\mathcal{B}^{[#1]}(#2)}
\sectionfont{\large \bfseries}
\subsectionfont{\normalsize}

\long\def\symbolfootnote[#1]#2{\begingroup%
\def\thefootnote{\fnsymbol{footnote}}\footnote[#1]{#2}\endgroup}

\def\blfootnote{\xdef\@thefnmark{}\@footnotetext}

\date{\today}

\begin{document}

{\Large \bfseries 
Geometric simplicial embeddings of arc-type graphs

}

{\large 
Hugo Parlier \symbolfootnote[1]{
Research partially supported by ANR/FNR project SoS, INTER/ANR/16/11554412/SoS, ANR-17-CE40-0033.\\
\vspace{.1cm} \\
{\em 2010 Mathematics Subject Classification:} Primary: 57M15. Secondary: 05C60. \\
{\em Key words and phrases:} Arc graphs, flip graphs, mapping class groups.}
and Ashley Weber}

\vspace{0.5cm}

{\bf Abstract.} 
In this paper, we investigate a family of graphs associated to collections of arcs on surfaces. These {\it multiarc graphs} naturally interpolate between arc graphs and flip graphs, both well studied objects in low dimensional geometry and topology. We show a number of rigidity results, namely showing that, under certain complexity conditions, that simplicial maps between them only arise in the ``obvious way". We also observe that, again under necessary complexity conditions, subsurface strata are convex. Put together, these results imply that certain simplicial maps always give rise to convex images. 

\vspace{1cm}
\section{Introduction}

Both arc graphs and flip graphs are useful objects for studying surfaces and their mapping class groups. The former is a Gromov hyperbolic graph on which the mapping class group acts, while the latter is, for large enough complexity, not Gromov hyperbolic but is always quasi-isometric to the underlying mapping class group. Arc graphs have single arcs as vertices while flip graphs have maximal sets of disjoint arcs (triangulations) as vertices. Thus interpolating in between them gives rise to a family of graphs where vertices are multiarcs (sets of arcs of a given size), and which we study in this article. We're interested in rigidity phenomena and to what extent results about the geometry of flip graphs extend to these graphs.

Before describing our results, let us point out that the setup is quite similar in nature to multi-curve graphs which interpolate between curve graphs and pants graphs. Erlandsson and Fanoni \cite{Erlandsson-Fanoni} studied their rigidity from the point of view of understanding simplicial maps between them. These generalized results of Aramayona who proved these results for pants graphs \cite{Aramayona}. For context, we note that mapping class groups also enjoy forms of strong rigidity, see for example the results of Aramayona and Souto \cite{Aramayona-Souto}. These results inspired us to study similar phenomena between multiarc graphs, as generalizations of similar results for flip graphs \cite{AKP}.

Another source of inspiration comes from questions about the intricate geometry of these combinatorial objects. For both pants graphs and flip graphs, the following question makes sense: if two pants decompositions or triangulations in a given graph have a curve or arc in common, does any geodesic between them always retain the common curve or arc? A geodesic between, say, two pants decomposition describes how to transform one into the other in the smallest possible number of moves, so it seems natural to never move a curve already in place. And in fact, for pants decompositions, this is at least {\it quasi} true, by results of Brock \cite{Brock} and Wolpert \cite{Wolpert}, by which we mean that you can find quasi-geodesics between the pants decompositions which do exactly that. However, despite partial results, it is not known in general \cite{APS1, APS2, ALPS, Taylor-Zupan} and is in fact equivalent to asking whether there are always a finite number of geodesics between vertices of the pants graph. For flip graphs, this is {\it always} known to be the case by recent results of Disarlo with the first author \cite{Disarlo-Parlier}. 

Our first results show that this continues to hold for multiarc graphs. For a surface $S$ and an integer $k\geq 1$, we denote by $\Acall^{[k]}(S)$ the $k$-multiarc graph (see Section \ref{sec:prelim} for the precise definition of vertices and edges). The subset of $\Acall^{[k]}(S)$ consisting of multiarcs which contain a given multiarc $\nu$ is denoted by $\arcgraphstratum{k}{\nu}{S}$. 

\begin{thm}\label{convexity main}
Let $k' \leq k$. For any $k'$-multiarc $\nu$, $\arcgraphstratum{k}{\nu}{S}$ is strongly convex.
\end{thm}

One expects this type of result to be true for the corresponding multicurve graphs, but nothing of the type is known. 

We then shift our focus to rigidity questions, as in the Erlandsson-Fanoni results alluded to earlier. We show that our graphs exhibit the same strong rigidity properties as the corresponding curve graphs. We state the main result in the following theorem, where non-exceptional just means that we disallow some low complexity cases (see Section \ref{sec:prelim} for the exact definition). 

\begin{thm}
\label{embedding main}
Let $S_1$ and $S_2$ be non-exceptional surfaces such that the complexity $\omega(S_1)$ is at least $7 + k_1$. Let $\phi: \arcgraph{k_1}{S_1} \hra \arcgraph{k_2}{S_2}$ be a simplicial embedding, with $k_2 \geq k_1$, and assume $\omega(S_2) - \omega(S_1) \leq k_2 - k_1$. Then $\omega(S_2) - \omega(S_1) = k_2 - k_1$ and 
\begin{itemize}
\item if $k_2 = k_1$, $\phi$ is an isomorphism induced by a homeomorphism $f: S_1 \ra S_2$; 
\item if $k_2 > k_1$, there exists a $\pi_1$-injective embedding $f: S_1 \ra S_2$ and a $(k_2-k_1)$-multiarc $\nu$ on $S_2$ such that for any $\mu \in \arcgraph{k_1}{S_1}$ we have 
\begin{equation*}
\phi(\mu) = f(\mu) \cup \nu.
\end{equation*}
\end{itemize} 
\end{thm}

Before stating more results, let us make a few comments. In some sense, the fact that these graphs all have the mapping class group as isomorphism group is a particular case of the above theorem (see Theorem \ref{rigidity}), but the above result is of course strictly stronger. A second remark is about our complexity conditions. While it is impossible to be completely free of them, it is unclear to what extent these conditions can be relaxed. In our approach, we use simplicial rigidity phenomena of flip graphs which already have complexity conditions built-in (although it is also unclear how necessary they are even in this case).

Finally, we put these results together to show that under certain complexity conditions, the image of simplicial maps is always geometric:

\begin{thm}
\label{main}
Let $S_1$ and $S_2$ be non-exceptional surfaces such that the complexity $\omega(S_1)$ is at least $7 + k_1$. Let $\phi: \arcgraph{k_1}{S_1} \hra \arcgraph{k_2}{S_2}$ be a simplicial embedding, with $k_2 \geq k_1$, and assume $\omega(S_2) - \omega(S_1) \leq k_2 - k_1$. Then $\phi\left(\arcgraph{k_1}{S_1}\right)$ is strongly convex inside $\arcgraph{k_2}{S_2}$.
\end{thm}

{\bf Acknowledgements.} This project started during a GEAR sponsored visit of the second author at the University of Luxembourg. Both authors acknowledge support from U.S. National Science Foundation grants DMS 1107452, 1107263, 1107367 RNMS: Geometric structures And Representation Varieties (the GEAR Network).

 \section{Preliminaries}\label{sec:prelim}
All surfaces are compact, connected, and orientable with a non-empty finite set of marked points. A boundary component is either an isolated marked point or a boundary curve which is required to have at least one marked point on it. 

An \textit{exceptional} surface is a surface of genus at most one with at most 3 boundary components. Note that the number of marked points on a boundary component is not relevant. For instance, a polygon with any number of vertices is an exceptional surface. 

%%Need definition of mapping class group somewhere.

A \textit{k-multiarc} is a collection of $k$ disjoint arcs. The maximal size of a multiarc is the \textit{complexity} $\omega(S) = 6g + 3b + 3p + q - 6$, where $g$ is the genus, $b$ is the number of boundary components, $p$ is the number of marked points on the interior of $S$, and $q$ is the number of marked points on the boundary of $S$. A maximal size multiarc is called a \textit{triangulation}. 

%%Need definition of flip graph with flip moves separately.

We define a $k$-multiarc graph, $\arcgraph{k}{S}$, associated to a surface $S$. The vertices of this graph are $k$-multiarcs. Two multiarcs are connected with an edge if they share a $(k-1)$-multiarc, $\nu$, and the remaining two arcs intersect minimally on $S \backslash \nu$. With this definition, $\arcgraph{1}{S}$ is the arc graph and $\arcgraph{\omega(S)}{S}$ is the flip graph.

There is another subgraphs of a $k$-multiarc graph which proves to be useful. If $v$ is a vertex of $\arcgraph{k}{S}$ then \textit{star} of $\nu$, denoted $\St ( \nu )$, is the subgraph spanned by $\nu$ and all its adjacent vertices, in other words the $1$-neighborhood of $\nu$. 

We end this preliminary section by observing that our graphs of choice are all connected. 
\begin{lem} For a surface $S$ and an integer $k$ satisfying $1 \leq k \leq \omega(S)$, the graphs 
$\Acall^{[k]}(S)$ are connected.
\end{lem}

\begin{proof}
We give a proof by induction on $k$. The base case, for $k=1$, is the connectivity of the arc graph, which is well known.

Assume $\Acall^{[k-1]}(S)$ is connected for every surface satisfying $\omega(S) \geq k-1$. 
Consider a surface $S$ with complexity at least $k$ and two $k$-multiarcs $\alpha = \{a_1, \ldots a_k \}$ and $\beta = \{b_1, \ldots, b_k\}$.
By assumption, there is a path $\alpha \backslash a_k = \gamma_0, \gamma_1, \gamma_2, \ldots, \gamma_n =\beta \backslash b_k$ in $\Acall^{[k-1]}(S)$. Let $c_0=a_k$ and $c_i = \gamma_i \backslash \gamma_{i-1}$ for $0<i\leq n$. 
We know $c_0$ is connected to $c_1$ in $\arcgraph{1}{S\backslash (\gamma_0 \cap \gamma_1)}$, say by the path $c_0 = d_0, d_1, \ldots d_{m_1} = c_1$. Then we create the path $\delta^0 = \{\delta^0_i\}_{i = 0}^{m_1}$, in $\arcgraph{k}{S}$, where $\delta^0_i = \gamma_0 \cup d_i$. Similarly, we define paths $\delta^i$ arising from a path connecting $c_i$ to $c_{i+1}$. Concatenating $\{\delta^i \}_{i = 0}^n$ gives a path between $\alpha$ and $\beta$ in $\arcgraph{k}{S}$.

%%%THIS IS SIMILAR TO WHAT'S IN VIVEKA AND FEDERICA'S PAPER%%%%%%%%%%%%
%There is a path $a_1 = c_0, c_1 \ldots c_n = b_1$ in $\Acal(S)$. Let $\gamma_1$ be a $k$-multiarc containing $a_1$ and $c_1$. Since $\Acal^{[k-1]}(S\backslash a_1)$ is connected there is a path between $\alpha \backslash a_1$ and $\gamma_1 \backslash a_1$, so there is a path in $\arcgraph{k}{S}$ connecting $\alpha$ and $\gamma_1$. 
%We can define multiarcs $\gamma_2, \ldots \gamma_{n}$ similarly, where $\gamma_i$ contains $c_{i - 1}$ and $c_i$ and there is a path $\gamma_i$ and $\gamma_{i+1}$ in $\arcgraph{k}{S}$, and a path between $\gamma_n$ and $\beta$. 
%Therefore we have a path between $\alpha$ and $\beta$ in $\arcgraph{k}{S}$. 
\end{proof}

%%%The ending needs to be worded better, but it's correct.
\section{Convexity}
In this section, we show that our graphs have nice convexity properties, namely that geodesics between multiarcs which share an arc also share this arc:

\begin{thm}
Let $k' \leq k$. For any $k'$-multiarc $\nu$, $\arcgraphstratum{k}{\nu}{S}$ is strongly convex.
\end{thm}

%%%%%%%%Defining the map %%%%%%%%%%%

\begin{proof}
Say the multiarcs $\alpha, \beta \in \arcgraph{k}{S}$ have exactly one arc in common, $x$. 
Consider a geodesic $g$ between $\alpha$ and $\beta$ in $\arcgraph{k}{S}$, where $g$ is $\alpha = \gamma_0, \gamma_1, \ldots, \gamma_n = \beta$. 
Our goal is to show $x \in \gamma_i$ for all $0 \leq i \leq n$. 
Assume not, then there exists at least one subpath $\gamma_j, \gamma_{j+1}, \ldots, \gamma_m$, $j >0$, $m < n$, where no multiarc in the subpath contains $x$. 

Assume $\left(\cup_{t=j}^m \gamma_t\right)\cap x = \emptyset$. There is one arc in $\gamma_m \backslash \gamma_{m+1}$, denote this arc as $b$; we say that $b$ is sent to $x$.
Counting backwards from $m$, let $i$ be the first number such that $b \notin \gamma_i$ but $b \in \gamma_{i+1}$, note $j \leq i \leq m-1$. 
Then $\gamma_i$ and $\gamma_{i+1}$ differ by one arc, $b'$ and $b$. Since $\gamma_i \cap x = \emptyset$, $b'$ is disjoint from $x$, so we can replace $\gamma_{t}$ with $\gamma '_{t} = \gamma_{t} \backslash b \cup x$ for all $t$ between $i+1$ and $m$. Consequently, $\gamma '_m = \gamma_{m+1}$, so this new path, $\alpha = \gamma_0, \ldots \gamma_i, \gamma '_{i+1}, \ldots \gamma '_{m}, \gamma_{m+2}, \ldots, \gamma_n = \beta$ is shorter than $g$, contradicting the fact that $g$ is a geodesic. 

Therefore we can assume $\left(\displaystyle\cup_{t=j}^m \gamma_t\right)\cap x \neq \emptyset$. Using the same argument as above, we may assume that $\gamma_{j+1}, \ldots, \gamma_{m-1}$ all contain an arc that intersects $x$. Pick an orientation of $x$ and denote $x^+$ as the oriented arc. Let $b$ be any other arc on $S$, define $\pi_{x^+}(b)$ as follows:
\begin{itemize}
\item if $i(x,b) = 0$ then $\pi_{x^+}(b) = b$
\item if $i(x,b) > 0$ then $\pi_{x^+}(b)$ is the collection of arcs obtained by {\it combing} $b$ along $x^+$, as shown in Figure \ref{combing arcs}. Note that each arc in $\pi_{x^+}(b)$ has at least one endpoint that coincides with $x$. 
\end{itemize} 
For a multiarc $\alpha = \{a_1, \ldots, a_n \}$, define $\pi_{x^+}(\alpha) = \cup_{i = 1}^n \pi_{x^+}(a_i)$.

Note that in \cite{Disarlo-Parlier} this same map is defined for triangulations, so when $k=\omega(S)$. 

%%FIGURE%%
\begin{figure}[h]
\centering
\includegraphics[trim = {0in, 6.2in, 1.3in, 3.3in}, clip, scale = .75]{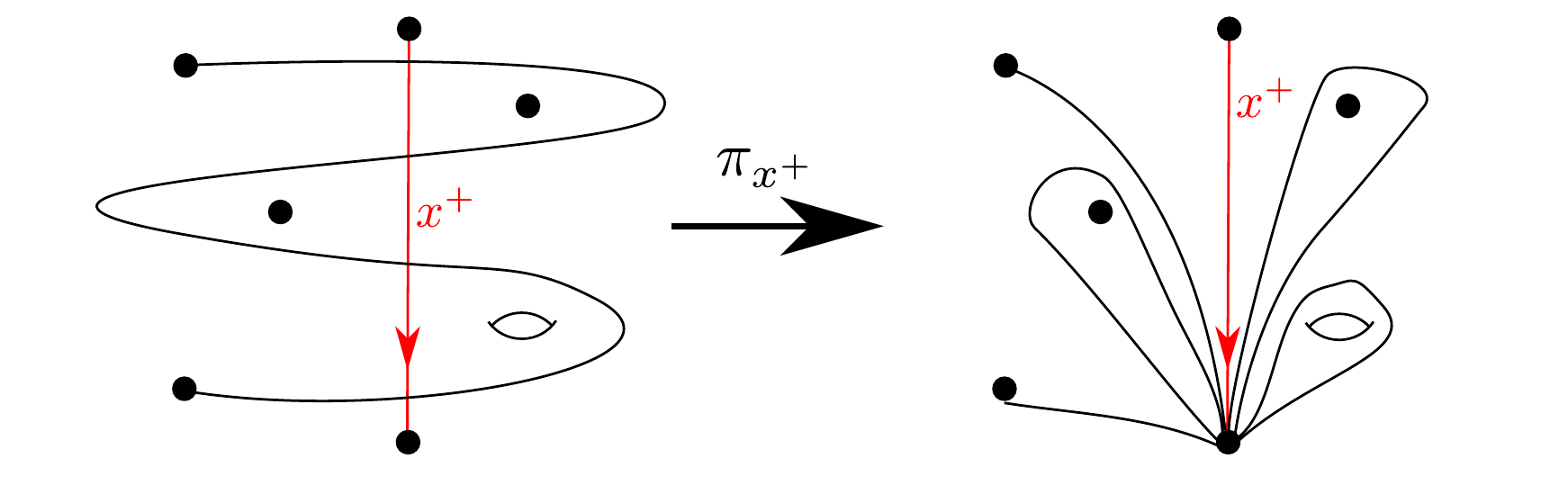} %{r, b, l, t}
\caption{An example of an arc being combed along $x^+$.}
\label{combing arcs}
\end{figure}

%\begin{lem}
%\label{Enough arcs}
%$|\pi_{x^+}(\alpha)| \geq k-1$.
%\end{lem}
%
%\begin{proof}

Let $\alpha = \{a_1, a_2, \ldots, a_k\}$, where $a_1$ is the first arc that intersects $x$ looking up from $x^+$, $a_t$ is the last arc that intersects $x$, and $a_i$, for $t < i \leq k$, does not intersect $x$. 
\begin{lem}
\label{one collapse}
The arc $a_1$ is the only arc in $\{a_1, \ldots a_t\}$ such that $\pi_{x^+}(a_1)$ is either contained in $\alpha$ or is peripheral. 
\end{lem}

\begin{proof}
Under the $\pi_{x^+}$ ma, $a_i$, for $1 < i \leq t$, is combed towards the endpoint of $x^+$. Since $a_1$ is below $a_i$, any arc created under the action of $\pi_{x^+}$ will intersect $a_1$ or be in $\pi_{x^+}(a_1)$, see Figure \ref{collapsing arcs}. Since $a_1$ and $a_i$ are distinct arcs $\pi_{x^+}(a_i) \neq \pi_{x^+}(a_1)$, so $\pi_{x^+}(a_i)$ must contain an arc that is not in $\alpha$ or peripheral. 
\end{proof}

Note, if $\pi_{x^+}(a_1)$ is either in $\alpha$ or is peripheral then we say $a_1$ {\it collapses}.

%%FIGURE%%
\begin{figure}[h]
\centering
\includegraphics[trim = {0in, 6in, 0in, 2.5in}, clip, scale = .75]{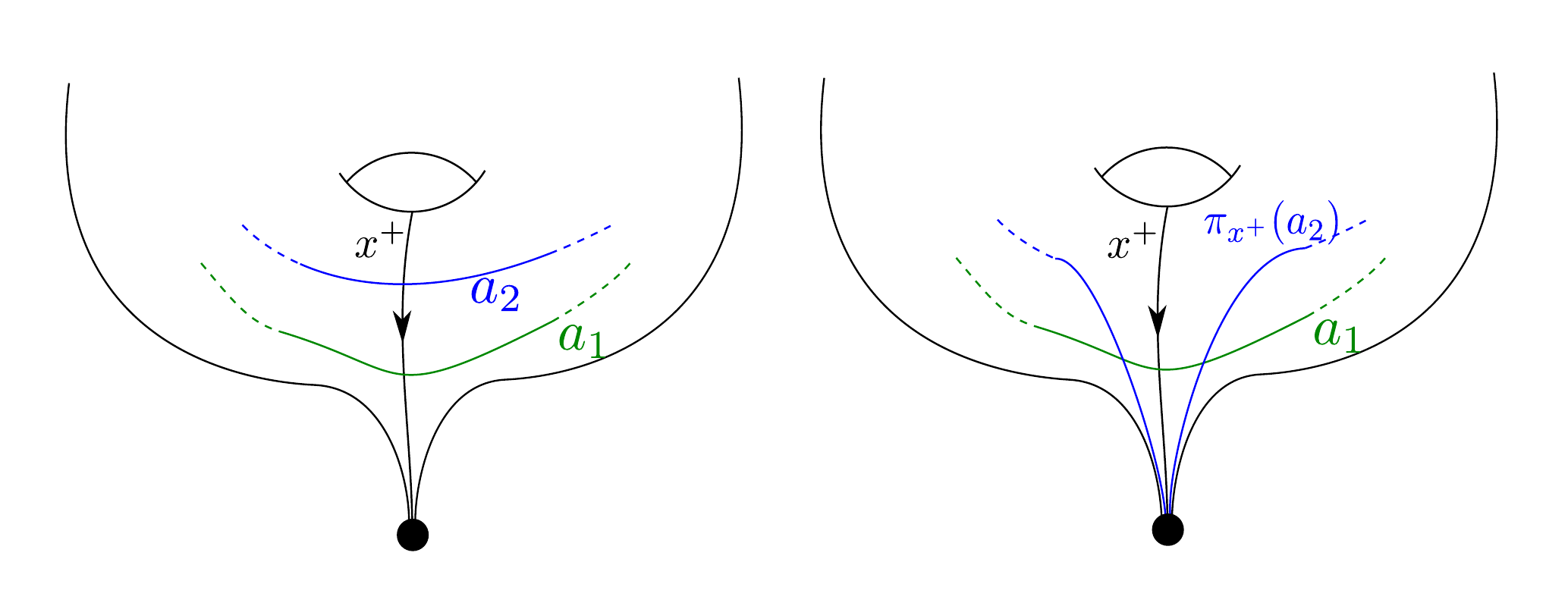}
\caption{}
\label{collapsing arcs}
\end{figure}

\begin{lem}
\label{assignment}
There is at least one way to assign each arc in $\alpha$ to an arc in $\pi_{x^+}(\alpha) \cup x$ such that the assignment is injective. Furthermore, an arc $a \in \alpha$ is assigned an arc in $\pi_{x^+}(a)$.
\end{lem}
\begin{proof}
First, complete the $k$-multiarc $\alpha$ to a triangulation $\alpha_{T}$ on $S$. As we have a triangulation, we are in the situation described in \cite{Disarlo-Parlier} and in this case $\pi_{x^+}(\alpha_T)$ contains $\omega(S) - 1$ arcs. It follows that one arc in $\alpha_T$ must collapse under $\pi_{x^+}$. 

We know by Lemma \ref{one collapse}, the first arc as we look up from $x^+$ in $\alpha_{T}$ is the arc which collapses; we assign this arc to $x$. Then every other arc in $\alpha_{T}$ can be assigned to an arc in $\pi_{x^+}(\alpha_T)$ such that the assignment is injective and for each arc $a \in \alpha_T$, $a$ is assigned an arc in $\pi_{x^+}(a)$. The assignments of the arcs in $\alpha$ define the assignment we want. 
\end{proof}

Recall the geodesic $g$ from $\alpha$ to $\beta$ and the subpath $\gamma_j, \gamma_{j+1}, \ldots, \gamma_m$. Assume we have modified the path with $\gamma_{i+1}', \ldots \gamma_{m}'$ where $\gamma_t'$ is an assignment as described in Lemma \ref{assignment}, for all $i+1 \leq t \leq m$, for some $i$ between $j$ and $m$. We define $\gamma_{i}'$ by assigning each arc in $\gamma_{i}$ to an arc in $\pi_{x^+}(\gamma_i)$ as follows. 
By definition of a path between $k$-multiarcs, $\gamma_{i}$ differs from $\gamma_{i+1}$ by one arc; call this arc $y$. 
Assume $y$ is sent to $z$ in $\gamma_{i+1}$. 
Assign each arc in $\gamma_{i}\backslash y$ as it was assigned in $f_{x^+}(\gamma_{i
+1})$. 
Then assign $y$ as follows: 
\begin{itemize}
\item If $\pi_{x^+}(y)$ is empty then assign $y$ to $x$. 
\item If $\pi_{x^+}(y)$ contains an arc not already assigned to an arc in $\gamma_i \backslash y$, then assign $y$ to that arc.
\item If all arcs in $\pi_{x^+}(y)$ are assigned to an arc in $\gamma_{i}\backslash y$ already then we need to reassign some arcs in $\gamma_{i} \backslash y$. 
Assign $y$ to an arc in $\pi_{x^+}(y)$, call it $y'$. Then find the arc in $\gamma_{i}$ assigned to $y'$ and reassign it to an arc in its image under $\pi_{x^+}$. Work backwards like this until all arcs in $\gamma_{i}$ are assigned properly; this is possible by Lemma \ref{assignment}.
\end{itemize}
These assignments define $\gamma_{i}'$.

Starting with $\gamma_{m-1}$, set $\gamma_{m-1}'$ to be the arcs in $\pi_{x^+}(\gamma_{m-1})$ assigned to the arcs in $\gamma_{m-1}$ as described in Lemma \ref{assignment}. Working backwards, define $\gamma_i'$ for $j \leq i \leq m-2$ using $\gamma_{i+1}'$ as described above. Now we have a valid path of the same length where $\left(\cup_{t=j}^n \gamma_t\right)\cap x = \emptyset$. Therefore, by the arguments above, this path is not a geodesic. The result follows.
\end{proof}
\section{Simplicial embeddings}

In this section we show, given surfaces under certain complexity conditions, that any simplicial map between arc graphs from the surface of largest complexity to the other is induced by a homeomorphism. We then promote this result to simplicial embeddings between arc graphs and multiarc graphs.

The following lemma is one of the key tools:
\begin{lem}
\label{permute}
If a complete subgraph contains 3 vertices that permute $k+1$ arcs (i.e. the intersection between the 3 vertices contains $k-2$ arcs) then every vertex in the complete subgraph permutes those $k+1$ arcs.
\end{lem}

\begin{proof}
Let $\alpha$, $\beta$, and $\gamma$ be the three vertices that permute $k+1$ arcs, let $\alpha = \{ a_1, \ldots a_k \}$. Denote by $\nu_i = \alpha \backslash a_i$. 
Without loss of generality, $\beta = \nu_k \cup b_k$ and $\gamma = \nu_j \cup b_k$ for some $j \neq k$. 
Take $\delta$ to be another multiarc in the complete subgraph, recalling $\delta$ is adjacent to $\alpha$, $\beta$, and $\gamma$. 
By contradiction, suppose that $\delta$ contains an arc $d$ not equal to any arc in $\alpha$, $\beta$, or $\gamma$. 
If $\delta$ contains $b_k$, then $\alpha$ and $\delta$ differ by two arcs, contradicting the fact that $\alpha$ and $\delta$ are adjacent. 
If $\delta$ does not contain $b_k$, then to be adjacent to $\beta$, $\delta = \nu_k \cup d$. In this case, $\gamma$ and $\delta$ differ by two arcs, contradicting the fact that $\beta$ and $\delta$ are adjacent. 
Therefore, $\delta = \nu_i \cup b_k$, for some $i \neq j, k$.
\end{proof}

The main trick in the following proof is the observation that a simplicial map between arc graphs induces a simplicial map between flip graphs. This is because triangulations correspond to maximal subgraphs in arc graphs and you can ``see" flips by seeing which maximal subgraphs share near maximal graphs. 

\begin{thm}
\label{base case}
Let $S_1$ and $S_2$ be two surfaces of complexity at least $7$ such that $\omega(S_2) \leq \omega(S_1)$. Let $\phi: \Acall(S_1) \hra \Acall(S_2)$ be a simplicial embedding. Then $\phi$ is an isomorphism induced by a homeomorphism $f: S_1 \ra S_2$.
\end{thm}

\begin{proof}
For any surface $S$ the maximum complete subgraph in $\Acall(S)$ is of size $\omega(S)$. Since $\phi$ is a simplicial embedding, a complete graph in $\Acall(S_1)$ is mapped to a complete graph of the same size in $\Acall(S_2)$. Therefore $\omega(S_1) \leq \omega(S_2)$, and by assumption $\omega(S_2) \leq \omega(S_1)$ so $\omega(S_1) = \omega(S_2)$. 

A maximum complete graph in $\Acall(S)$ corresponds to a triangulation, so we know that triangulations in $\Acall(S_1)$ are mapped to triangulations in $\Acall(S_2)$. 
%%Might want to say more about flip moves being preserved.
Flip moves are also preserved since they correspond to two maximum complete graphs that intersect in $\omega(S_1) - 1$ vertices. 
Hence $\phi$ induces a map $\overline \phi: \Fcal(S_1) \hra \Fcal(S_2)$ between flip graphs that is a simplicial embedding. By \cite{Korkmaz-Papadopoulos} there is a homeomorphism $f: S_1 \ra S_2$ that induces $\overline \phi$, and therefore $f$ induces $\phi$. 
%Might want to say more about why if f induces $\overline \phi$ it also induces $\phi$.
\end{proof}

The following result concerns topological types of arcs, and will be a tool that we use later. Two arcs, $a_1$ on $S_1$ and $a_2$ on $S_2$ are of same topological type if there is homeomorphism $\varphi:S_1\to S_2$ sending $a_1$ to $a_2$. As we've just seen that simplicial embeddings under certain complexity conditions come from homeomorphisms, we immediately get the following. 
\begin{coro}\label{arcs}
Let $S_1$ and $S_2$ be two surfaces of complexity at least $7$ such that $\omega(S_2) \leq \omega(S_1)$. Let $\phi: \Acall(S_1) \hra \Acall(S_2)$ be a simplicial embedding. Then $\phi$ sends arcs of $S_1$ to arcs of the same topological type. 
\end{coro}

An {\it ear} on a surface $S$ is an arc $e$ such that $S\setminus e$ has a triangle as a connected component. A non-separating arc is an arc $a$ such that $S\setminus a$ is non-separating. So a particular consequence of the above result is that ears are sent to ears and non-separating arcs to non-separating arcs.

We now prove the following result which generalizes the above results to embeddings of arc graphs into multiarc graphs. 

\begin{thm}
\label{1 to k}
Let $S_1$ and $S_2$ be two non-exceptional surfaces of complexity at least $7$ such that $\omega(S_2) - \omega(S_1) \leq k-1$ and $k > 1$. Let $\phi: \Acall(S_1) \hra \arcgraph{k}{S_2}$ be a simplicial embedding. Then $\omega(S_2) - \omega(S_1) = k-1$ and there exists a $\pi_1$ injective embedding $f: S_1 \ra S_2$ and a $(k-1)$-multiarc $\nu$ on $S_2$ such that for any $\mu \in \arcgraph{1}{S_1}$ we have $\phi(\mu) = f(\mu) \cup \nu$.
\end{thm}

\begin{proof}
Take $a \in \arcgraph{1}{S_1}$ and let $\alpha = \phi(a) = \{ a_1, \ldots a_k \}$. Define $\nu_j = \{a_1, \ldots, \hat a_j, \ldots, a_k \}$.
Our goal is to show $\phi(\St(a)) \subset \arcgraphstratum{k}{\nu_k}{S_2}$. Take $b, c \in \arcgraph{1}{S_1 \backslash a}$ such that they are disjoint; since the complexity is greater than $6$ these arcs exists. Without loss of generality we can assume $\beta = \phi(b) = \{a_1, \ldots, a_{k-1}, b_k \}$. Now there are two cases:
\begin{enumerate}
\item $\gamma = \phi(c) = \nu_j \cup b_k$ where $j \neq k$, or
\item $\gamma = \phi(c) = \nu_k \cup c_k$, where $c_k \in \arcgraph{1}{S_2 \backslash \nu_k}$.

\end{enumerate}
%%Do we have to explain anything about why these imply this?

If Case 1 happens: Take any arc $d \in \arcgraph{1}{S_1}$ such that $d$ is disjoint from $a$, $b$, and $c$; this exists because $\omega(S_1)\geq 7$. Then consider $\phi(d)$. Since $\phi$ is a simplicial embedding $\phi(d)$ is mutually disjoint from $\alpha$, $\beta$, and $\gamma$. By Lemma \ref{permute} $\phi(d) = \nu_i \cup b_k$ for some $i \neq j, k$. If there are $k-1$ arcs in $\arcgraph{1}{S_1 \backslash (a \cup b \cup c)}$ then injectivity of $\phi$ would be violated by taking $k-1$ arcs and applying Lemma \ref{permute}. By our restrictions on the surfaces, $S_1 \backslash (a \cup b \cup c)$ is not a collection of polygons (surfaces with exactly one boundary component, no genus, and no interior marked points). Thus, there are infinitely many arcs in $\arcgraph{1}{S_1 \backslash (a \cup b \cup c)}$ because $\omega(S_1) \geq 7$. 
Therefore, case 2 must happen. 

Take any arc in $\St(a)$ different from $b$ or $c$; let this arc be $d$. Then one can find a path from $c$ to $d$ in $\arcgraph{1}{S_1 \backslash a}$, $c = c_0, c_1, c_2, \ldots c_m = d$. 
For all $0 \leq i \leq m$, the arcs $c_i$, $c_{i+1}$, and $a$ form a triangle and by the contradiction in case 1, $\phi(c_{i}) = \nu_k \cup x$ for some $x \in \arcgraph{1}{S_1}$. Therefore $\phi(St(a)) \subset \arcgraphstratum{k}{\nu_k}{S_2} $. 

Now we have the following map diagram:

\begin{center}
\begin{tikzpicture}
 \node (E) at (0,0) {$\arcgraph{1}{S_1}$};
 \node[right=of E] (F) {$\arcgraphstratum{k}{v_k}{S_2} \subset \arcgraph{k}{S_2}$};
 \node[below=of F] (N) {$\arcgraph{1}{S_2 \backslash \nu_k}$};
 
 \draw[right hook->] (E)--(F) node [midway,above] {$\phi$};
 \draw[->] (F)--(N) node [midway,right] {$\cong \theta$};
 \draw[dashed,->] (E)--(N) node [midway,below] {$\Phi$};
 
\end{tikzpicture}
\end{center}

where $\theta$ is an isomorphism defined by $\theta(\nu_k \cup x)=x$ and $\Phi = \theta \circ \phi$. Since $\phi$ is a simplicial embedding, $\Phi$ must be a simplicial embedding as well; so $\omega(S_2 \backslash \nu_k) \leq \omega(S_1)$. By Theorem \ref{base case}, $\Phi$ is induced by a homeomorphism $f: S_1 \ra S_2 \backslash \nu_k$. Composing $f$ with the natural inclusion $S_2 \backslash \nu_k \hra S_2$, we get a $\pi_1$-injective map $F: S_1 \ra S_2$ where $\phi$ is induced by $F$ and $\nu_k$.

\end{proof}

\section{Rigidity}
We now proceed to show that our graphs all have the expected automorphism groups, completing the well-known results for flip graphs and the arc graph.

\begin{thm}
\label{rigidity}
$\Aut(\arcgraph{k}{S}) = \Mod(S)$.
\end{thm}

We will prove Theorem \ref{rigidity} by induction. It is known that $\Aut(\arcgraph{1}{S}) \simeq \Mod(S)$ \cite{Disarlo, Irmak-McCarthy, Korkmaz-Papadopoulos}. 
Assume $\Aut(\arcgraph{k-1}{S}) \simeq \Mod(S)$. We will show $\Aut(\arcgraph{k}{S}) \simeq \Aut(\arcgraph{k-1}{S})$.
 Define the following map:
\begin{align*}
\theta: E(\arcgraph{k}{S}) &\lra V(\arcgraph{k-1}{S}) \\
\alpha \beta &\longmapsto \alpha \cap \beta
\end{align*}

By definition of a multiarc graph, it follows that $\theta$ is surjective. 

\begin{lem}
\label{well defined}
For $A \in \Aut(\arcgraph{k}{S})$, if for $ef \in E(\arcgraph{k}{S})$ $\theta(e) = \theta(f)$, then $\theta(A(e)) = \theta(A(f))$.
\end{lem}

\begin{proof}
By definition, $\theta(e) = \theta(f) = \mu$ if and only if $e$ and $f$ are in $\arcgraphstratum{k}{\mu}{S} \subset \arcgraph{k}{S}$. Therefore $A(e), A(f) \in A(\arcgraphstratum{k}{\mu}{S})$. A induces a map:
\begin{align*}
\psi: \arcgraph{1}{S \backslash \mu} &\lra \arcgraph{k}{S} \\
a &\longmapsto A(\{ a\} \cup \mu)
\end{align*}
which is a simplicial embedding. %%Do we have to say something about why?
Therefore, by Theorem \ref{1 to k} there exists a $(k-1)$-multiarc $\nu$ such that $\psi(\arcgraph{1}{S \backslash \mu}) \subset \arcgraphstratum{k}{\nu}{S}$, which implies $\theta(A(e)) = \nu = \theta(A(f))$.
\end{proof}

Define the map $\phi$ as follows:
\begin{align*}
\phi(A): V(\arcgraph{k-1}{S}) &\lra V(\arcgraph{k-1}{S}) \\
\mu &\longmapsto \theta(A(e))
\end{align*}
where $e$ is any edge such that $\theta(e) = \mu$. $\phi(A)$ is well defined by Lemma \ref{well defined} and is a bijective map because its inverse is $\phi(A\inv)$.

\begin{lem}
$\phi(A)$ sends edges to edges.
\end{lem}

\begin{proof}
Take an edge, $\mu \nu$, in $\arcgraph{k-1}{S}$. Then $\arcgraphstratum{k}{\mu}{S} \cap \arcgraphstratum{k}{\nu}{S} = \{ \mu \cup \nu\}$, therefore $A(\arcgraphstratum{k}{\mu}{S})$ and $A(\arcgraphstratum{k}{\nu}{S})$ intersect in exactly one vertex. This implies $\arcgraphstratum{k}{\phi(A)(\mu)}{S} \cap \arcgraphstratum{k}{\phi(A)(\nu)}{S}$ is non empty. 
Since $\phi(A)(\nu)$ and $\phi(A)(\mu)$ are $k-1$ multiarcs, either $\arcgraphstratum{k}{\phi(A)(\mu)}{S} = \arcgraphstratum{k}{\phi(A)(\nu)}{S}$ or $\arcgraphstratum{k}{\phi(A)(\mu)}{S} \cap \arcgraphstratum{k}{\phi(A)(\nu)}{S}$ is just one vertex. 
If $\arcgraphstratum{k}{\phi(A)(\mu)}{S} = \arcgraphstratum{k}{\phi(A)(\nu)}{S}$ then $\phi(A)(\mu) = \phi(A)(\nu)$, violating the injectivity of $\phi$. 
So, $\arcgraphstratum{k}{\phi(A)(\mu)}{S} \cap \arcgraphstratum{k}{\phi(A)(\nu)}{S} = \{\phi(A)(\mu) \cup \phi(A)(\nu) \}$. This tells us that $\phi(A)(\mu)\phi(A)(\nu)$ is an edge in $\arcgraph{k-1}{S}$.
\end{proof}

Thus, $\phi(A)$ is a simplicial map from $\arcgraph{k-1}{S}$ to itself. Furthermore, since it is also a bijection, $\phi(A)$ is an automorphism. Therefore we can define the map $\phi$:
\begin{align*}
\phi: \Aut(\arcgraph{k}{S}) &\lra \Aut(\arcgraph{k-1}{S}) \\
A &\longmapsto \phi(A)
\end{align*}

\begin{lem}
$\phi$ is a group homomorphism.
\end{lem}

\begin{proof}
Take $A, B \in \Aut(\arcgraph{k}{S})$ and a $(k-1)$ multiarc, $\mu$, say $\theta(e) = \mu$. 
\begin{align*}
\phi(A \circ B) = \theta(A \circ B(e)) = \theta(A(B(e))) = \phi(A)(\theta(B(e))) \\
= \phi(A)\phi(B)(\theta(e)) = \phi(A) \circ \phi(B) (\mu)
\end{align*} 
\end{proof}

\begin{lem}
$\phi$ is injective.
\end{lem}

\begin{proof}
We show $\ker(\phi) = \{\text{id}\}$. Assume $\phi(A) = \text{id}$. Take $\mu \in \arcgraph{k-1}{S}$ where $\mu = \{ m_1, \ldots, m_k\}$ and set $\mu_i = \mu \backslash m_i$. Define the edge $e_i$ to be $\mu\nu_i$ for some $\nu_i \in \arcgraph{k-1}{S}$ such that $\theta(e_i) = \mu_i$. Now
\begin{align*}
\mu_i = \phi(A)(\mu_i) = \theta(A(e_i)) = A(\mu) \cap A(\nu_i).
\end{align*} 
So $\mu_i \subset A(\mu)$ for all $1 \leq i \leq k$. Therefore, $A(\mu) = \mu$ and $A = \text{id}$.
\end{proof}

\begin{lem}
$\phi$ is surjective.
\end{lem}

\begin{proof}
Take the maps $F: \Mod(S) \ra \Aut(\arcgraph{k}{S})$ and $G: \Mod(S) \ra \Aut(\arcgraph{k-1}{S})$ to be the natural maps between the mapping class group and the multiarc graph. By the induction hypothesis, $G$ is surjective. Therefore, to show $\phi$ is surjective all that needs to be shown is $G = \phi \circ F$. Take $f \in \Mod(S)$ and a $(k-1)$ multiarc, $\mu = \{ m_1, \ldots, m_{k-1}\}$. Let $e = \alpha \beta$ be an edge in $\arcgraph{k}{S}$ such that $\theta(e) = \mu$. Then we have the following:
\begin{align*}
\phi(F(f))(\mu) &= \theta(F(f)(\mu)) = \theta(F(f)(\alpha) F(f)(\beta)) = F(f)(\alpha) \cap F(f)(\beta) \\
&= \{f(m_1), \ldots, f(m_{k-1}), f(a) \} \cap \{f(m_1), \ldots, f(m_{k-1}), f(b) \} \\
&= \{f(m_1), \ldots, f(m_{k-1})\} = G(f)(\mu)
\end{align*} 
\end{proof}

Therefore, $\Aut(\arcgraph{k}{S}) \simeq \Aut(\arcgraph{k-1}{S}) \simeq \Mod(S)$, proving Theorem \ref{rigidity}.

\section{Simplicial embedding between arc-type graphs}
%%I put this here because we need to use it, however it should probably end up earlier in the paper.
In this section, we prove our main theorem about simplicial embeddings. 

The following notions will come in handy. An {\it ear} on a surface $S$ is a separating arc such that one of the two connected components is a triangle. We say that two disjoint non-separating arcs or ears $a$ and $b$ form a \textit{nice pair} if $S \backslash (a \cup b)$ has exactly one component with positive complexity. 

We introduce a new subgraph of $\arcgraph{k}{S}$, $\B{k}{S}$, defined as follows:
\begin{align*}
V(\B{k}{S}) &= \{\mu \text{ } | \text{ } \mu \text{ contains a non-separating arc or ear} \} \\
E(\B{k}{S}) &= \{\mu \nu \text{ }| \text{ } |\mu \cap \nu| = k-1 \text{ and the remaining arcs form a nice pair} \}
\end{align*}

\begin{lem}
For $\omega(S) \geq 7$, $\B{1}{S}$ is connected.
\end{lem}

\begin{proof}
Assume $S$ is a punctured sphere, $S_{0, p}$ with $p \geq 4$. (if $p = 3$ then all disjoint non-separating arcs form nice pairs). For an arc to be non-separating in $S_{0,p}$ it must have two separate endpoints. 
Two disjoint non-separating arcs in $S_{0, p}$ form a nice pair if and only if the arcs don't share the same endpoints. Take non-separating arcs $a$ and $b$ such that they do not form a nice pair. Then they must have the same endpoints. 
We know both components of $S \backslash (a\cup b)$ is a surface containing a puncture that's not an endpoint of $a$ or $b$. Then there exists an arc, $c$, disjoint from both $a$ and $b$ and doesn't share both endpoints with $a$ or $b$. Therefore $a$ and $c$, and $b$ and $c$, form a nice pair. 
So we have the path $a, c, b$ in $\B{1}{S}$. 

Assume $S$ is a surface that's not a punctured sphere, therefore $S$ has positive genus. Here it is still true that if two disjoint non-separating arcs don't form a nice pair then they share both endpoints. Take non-separating disjoints arcs $a$ and $b$ such that they do form a nice pair. If $p \geq 3$ then one of the components of $S \backslash (a\cup b)$ has a puncture which is not an endpoint of $a$ or $b$ and we can find a non-separating arc $c$ with one endpoint on this puncture which is disjoint from $a$ and $b$. Therefore we have a desired path. If $p = 2$ then we can find a non-separating arc, $c$, in one of the components of $S \backslash (a\cup b)$ which starts and ends at the same puncture. Therefore $a$ and $c$, and $b$ and $c$, form a nice pair and we get the desired path. If $p=1$ then all disjoint non-separating arcs form nice pairs. 

Now assume $S$ is any surface. Take arcs $a, b \in \B{1}{S}$. Then we can form a unicorn path between $a$ and $b$:
$$
a = c_0, c_1, \ldots, c_n = b,
$$
being sure to pick the orientation of $a$ and $b$ such that $c_1$ has two distinct endpoints. We know that $c_i c_{i+1}$ don't form a nice pair for $1 \leq i \leq n-2$, however they are disjoint non-separating curves so we can connect each $c_i c_{i+1}$ in $\B{1}{S}$ as described above.
\end{proof}

\begin{remk}
\label{sep next to nonsep}
The following observations about a multiarc $\mu \in \arcgraph{k}{S} \backslash \B{k}{S}$ will be very useful:
\begin{itemize}
\item If $S$ is a surface with genus at least $1$ then there is one non-separating arc, $a$, that is disjoint from $\mu$. We can then Dehn twist $a$ around the meridian of one of the genuses. This results in infinitely many multiarcs in $\B{k}{S}$ adjacent to $\mu$.

\item If $S = S_{0,p}$ is a punctured sphere then there are at most $2p-5$ disjoint separating arcs, so $k \leq 2p-5$. Then there are at least $p-1$ non-separating arcs disjoint from $\mu$.  Therefore, there are $k(p-1)$ multiarcs in $\B{k}{S}$ adjacent to $\mu$.
%%Could possibly add more details

\item If $S$ has at least three boundary components and no genus then any separating arc must isolate one of the two boundary components and therefore $S \backslash \mu$ has a component with one of the boundaries in $S$. Take the curve circling the boundary; one can Dehn twist any non-separating arc disjoint from $\mu$ around the curve to get infinitely many non-separating arcs disjoint from $\mu$.

%\item If $S$ has only one boundary component then there must exist at least one puncture on the interior of $S$, otherwise we would be in the polygon case; denote one such puncture as $x$. Any arc from $x$ to a puncture on the boundary is non-separating. Thus, there are $p$ non-separating arcs in a triangulation of $S$. If $k > \xi(S) - p$ then $\B{k}{S} = \arcgraph{k}{S}$. If $k \leq \xi(S) - p$ and $p \geq 2$ then there are $p*k$ multiarcs adjacent to $\mu$ in $\B{k}{S}$. If $k \leq \xi(S) - p$ and $p = 1$, then we can find at least one ear and one non-separating arc from $x$ to a boundary puncture disjoint from $\mu$. Therefore, there are at least $2k$ multiarcs adjacent to $\mu$ in $\B{k}{S}$.
\end{itemize}
\end{remk}

\begin{thm}
\label{embedding main}
Let $S_1$ and $S_2$ be non-exceptional surfaces such that the complexity $\omega(S_1)$ is at least $7 + k_1$. Let $\phi: \arcgraph{k_1}{S_1} \hra \arcgraph{k_2}{S_2}$ be a simplicial embedding, with $k_2 \geq k_1$, and assume $\omega(S_2) - \omega(S_1) \leq k_2 - k_1$. Then $\omega(S_2) - \omega(S_1) = k_2 - k_1$ and 
\begin{itemize}
\item if $k_2 = k_1$, $\phi$ is an isomorphism induced by a homeomorphism $f: S_1 \ra S_2$; 
\item if $k_2 > k_1$, there exists a $\pi_1$-injective embedding $f: S_1 \ra S_2$ and a $(k_2-k_1)$-multiarc $\nu$ on $S_2$ such that for any $\mu \in \arcgraph{k_1}{S_1}$ we have 
\begin{equation*}
\phi(\mu) = f(\mu) \cup \nu.
\end{equation*}
\end{itemize} 
\end{thm}
%%Viveka and Federica have the complexity of S_1 is at least 4+k_1 - I don't know where this is used though...

The proof of Theorem \ref{embedding main} proceeds by induction on $k_1$. When $k_1 = 1$, Theorem \ref{1 to k} gives us the result. Now assume that we know the result up to $k_1-1$. 

If $a$ is an ear in $S_1$ then we denote $S_1 \backslash a$ to be the component of positive complexity. Now take a non-separating arc or ear $a$ on $S_1$. We can define the following map:
\begin{equation*}
\phi_a: \arcgraph{k_1-1}{S_1 \backslash a} = \arcgraphstratum{k_1}{a}{S_1} \hra \arcgraph{k_2}{S_2}.
\end{equation*}
By the induction hypothesis $\phi_a$ is induced by a $\pi_1$-injective embedding $f_a: S_1 \backslash a \ra S_2$ and 
\begin{align*}
k_2 - k_1 +1 &= \omega(S_2) - \omega(S_1 \backslash a) \\
	&= \omega(S_2) - \omega(S_1) + 1
\end{align*}
therefore $\omega(S_2) - \omega(S_1) = k_2 - k_1 =: d$, as desired. The induction hypothesis also states that along with $f_a$ there is a $(d+1)$-multiarc, $\nu_a$, such that $\phi_a(\mu) = f_a(\mu) \cup \nu_a$. 

Note that a $0$-multiarc is just the empty set. 

\begin{lem}
\label{on nice pairs}
If $a$ and $b$ are disjoint non-separating arcs or ears and form a nice pair on $S_1$ then there exists a $\pi_1$-injective embedding $f: S_1 \ra S_2$ and a $d$-multiarc $\nu$ on $S_2$ such that $\phi$ is induced by $f$ and $\nu$ on $\arcgraphstratum{k_1}{a}{S_1} \cup \arcgraphstratum{k_1}{b}{S_1}$.  Moreover, if $d = 0$, then $f$ is a homeomorphism.
\end{lem}

\begin{proof}
We know we have maps $\phi_a$, which is induced by $\nu_a$ and $f_a$ on $\arcgraphstratum{k_1}{a}{S_1}$, and $\phi_b$, which is induced by $\nu_b$ and $f_b$ on $\arcgraphstratum{k_1}{b}{S_1}$. What we will show is that $\nu_a \cap \nu_b$ is a $d$-multiarc and $f_a$, $f_b$ define a $\pi_1$-injective embedding $S_1 \ra S_2$ which is a homeomorphism when $d=0$. 

\textbf{Case $k_1 = 2$:} We have:
\begin{equation*}
\phi(\{a, b \}) = \nu_a \cup \{f_a(b) \} = \nu_b \cup \{ f_b(a)\}
\end{equation*}
which implies that $\nu_a \cap \nu_b$ is a $d$-multiarc if and only if $\nu_a \neq \nu_b$. 

To show $\nu_a \neq \nu_b$ we proceed by contradiction; assume $\nu_a = \nu_b$. Then 
\begin{align*}
\phi(\arcgraphstratum{2}{a}{S_1} \simeq \arcgraph{1}{S_1 \backslash a}) &\subseteq \arcgraphstratum{k_2}{\nu_a}{S_2} =\arcgraph{1}{S_2 \backslash \nu_a} \\ 
\phi(\arcgraphstratum{2}{b}{S_1} \simeq \arcgraph{1}{S_1 \backslash b}) &\subseteq \arcgraphstratum{k_2}{\nu_b}{S_2} =\arcgraphstratum{k_2}{\nu_a}{S_2} =\arcgraph{1}{S_2 \backslash \nu_a} 
\end{align*}
and since $\omega(S_1 \backslash a) = \omega(S_2 \backslash \nu_a)$ we can apply Theorem \ref{base case} to get $\phi(\arcgraphstratum{2}{a}{S_1}) = \phi(\arcgraphstratum{2}{b}{S_1}) = \arcgraphstratum{k_2}{\nu_a}{S_2}$. Now take $\mu \in \arcgraphstratum{k_2}{\nu_a}{S_2}$ where $\mu \neq \phi(\{a, b \})$, then there exists $\alpha \in \arcgraphstratum{2}{a}{S_1}$ and $\beta \in \arcgraphstratum{2}{b}{S_1}$ such that $\phi(\alpha) = \phi(\beta) = \mu$. Since $\phi$ is injective $\alpha = \beta$ giving $\alpha, \beta \in \arcgraphstratum{2}{a}{S_1} \cap \arcgraphstratum{2}{b}{S_1}$. However, there is only one multiarc in $\arcgraphstratum{2}{a}{S_1} \cap \arcgraphstratum{2}{b}{S_1}$ which is $\{a, b\}$, therefore $\alpha = \beta = \{a,b\}$ contradicting that $\mu = \phi(\{a,b\})$. So $\nu_a$ and $\nu_b$ are not equal and $\nu := \nu_a \cap \nu_b$ is a $d$-multiarc.

Now we will show $f_a$ and $f_b$ induce the same map on $S_1 \backslash (a \cup b)$. Take $c \in \arcgraph{1}{S_1\backslash (a \cup b)}$ and consider the multiarcs $\{a,c\}$ and $\{b,c \}$. These form an edge in $\arcgraph{2}{S_1}$, therefore $\phi(\{a,c \})\phi(\{b,c\})$ is an edge in $\arcgraph{k_2}{S_2}$. We know 
\begin{align*}
\phi(\{a,c\}) &= \nu_a \cup f_a(c) \\
\phi(\{b,c\}) &= \nu_b \cup f_b(c)
\end{align*}
and since $\nu_a \neq \nu_b$, $f_a(c) = f_b(c)$. Accordingly, $f_a$ and $f_b$ induce the same map $\arcgraph{1}{S_1 \backslash(a \cup b)} \ra \arcgraph{1}{S_2 \backslash \phi(\{a,b\})}$. Since the complexities of $S_1 \backslash (a\cup b)$ and $S_2 \backslash \phi(\{a,b\})$ are equal we may assume $g_a = f_a|_{S_1 \backslash (a \cup b)}$ and $g_b = f_b|_{S_1 \backslash (a \cup b)}$ are homeomorphisms onto $S_2 \backslash \phi(\{ a,b\})$. 
Moreover, $g_b \inv \circ g_a$ induces the identity map on $S_1 \backslash (a \cup b)$, therefore the class of $g_b \inv \circ g_a$ is trivial in $\Mod(S_1 \backslash (a \cup b))$. This implies $g_a = g_b \circ h$ where $h$ is some homeomorphism on $S_1 \backslash (a \cup b)$ isotopic to the identity. Thus we can assume $f_a$ and $f_b$ agree on $S_1 \backslash (a \cup b)$ and we define a map $f: S_1 \ra S_2$ by extending either $f_a$ or $f_b$ to $S_1$. The resulting map $f$ is a $\pi_1$-injective embedding. Hence $f$ induces a simplicial map between the arc graphs of $S_1$ and $S_2$.

When $d=0$, i.e. $k_1 = k_2 = 2$, $f(a) = \nu_a$ and $f(b) = \nu_b$ so $\nu_a$ and $\nu_b$ are both arcs. By Corollary \ref{arcs}, the arcs $\nu_a$ and $\nu_b$ are of the same topological type as $a$ and $b$. Now $f_a$ induces a simplicial map between the arc graph of $S_1\backslash (a \cup b)$ and the arc graph of $S_2\backslash (\nu_a \cup \nu_b)$. As the surfaces have the same complexity, by the rigidity of arc graphs, we necessarily have that $S_1\backslash (a \cup b)$ and $S_2\backslash (\nu_a \cup \nu_b)$ are homeomorphic and in fact $f_a$ is induced by a homeomorphism. The same argurment holds for $f_b$. As $f_a$ and $f_b$ naturally extend to $f$, and because the arcs $a$, resp. $b$, and $f(a) = \nu_a$, resp. $f(b) = \nu_b$, are of the same type, we can promote $f$ to a homeomorphism which completes the case when $k_1 = k_2 = 2$.

\textbf{Case $k_1 \geq 3$:} Take $\mu \in \arcgraphstratum{k_1}{a \cup b}{S_1}$, we know $\mu = \{a\} \cup \{b\} \cup \tmu$ where $\tmu$ is a $(k_1-2)$-multiarc. Then
\begin{align*}
\phi(\mu) &= \nu_a \cup f_a(\mu \backslash a) = \nu_a \cup \{ f_a(b) \} \cup f_a(\tmu) \\
	&= \nu_b \cup f_b(\mu \backslash b) = \nu_b \cup \{ f_b(a) \} \cup f_b(\tmu)
\end{align*}
As $\tmu$ varies, $f_a(\tmu)$ and $f_b(\tmu)$ varies since $\phi$ is injective and the rest of $\phi(\mu)$ is fixed since it does not depend on $\tmu$. 
Therefore $\nu_a \cup \{f_a(b) \} = \nu_b \cup \{ f_b(a) \}$ and $f_a(\tmu) = f_b(\tmu)$, implying $f_a$ and $f_b$ induce the same map $\arcgraph{k_1-2}{S_1\backslash(a \cup b)} \ra \arcgraph{k_1-2}{S_2 \backslash (\nu_a \cup f_a(b))}$. 
Now using the same argument as before, $f_a|_{S_1 \backslash (a \cup b)} = f_b|_{S_1 \backslash (a \cup b)} \circ h$ where $h$ is a homeomorphism on $S_1\backslash (a \cup b)$ isotopic to the identity. 
And as before, we can assume $f_a$ and $f_b$ agree on $S_1\backslash (a \cup b)$ and defines a map $f: S_1 \ra S_2$. The same argument as in the case where $k=2$ holds here as well, implying $f$ is a $\pi_1$-injective map and when $d=0$ $f$ is a homeomorphism. Finally, since $f$ is injective $f(a) \neq f(b)$, thus $\nu_a \cap \nu_b$ is a $d$-multiarc.
\end{proof}

\begin{lem}
\label{On B}
The embedding $\phi$ is induced by $f$ and $\nu$ on $\B{k}{S_1}$.
\end{lem}

\begin{proof}
Let $f$ and $\nu$ be determined by the nice pair $(a,b)$ as in Lemma \ref{on nice pairs}. Take an arc $c \in \B{1}{S_1}$, since $\B{1}{S_1}$ is connected there exists a path 
\begin{equation*}
a = c_0, c_1, \ldots c_m = c.
\end{equation*}
As $a$ and $c_1$ form a nice pair, by Lemma \ref{on nice pairs} they determine a map $f'$ and a $d$-multiarc $\nu'$. Both $f$ and $f'$ agree with $f_a$ so we can assume that $f' = f$ on $S_1 \backslash a$. Then $f'(\mu) \cup \nu' = \phi(\mu) = f(\mu) \cup \nu$ for any multiarc $\mu$ on $S_1 \backslash a$, therefore $\nu' = \nu$. 

If $a$ is a non-separating arc then $f=f'$ on $S_1$ by continuity. Otherwise, if $a$ is an ear then $f =f'$ on everything but a triangle, but two maps can't differ on just a triangular portion, so they must be equal on all of $S$.
Now, repeating the argument along the path we know that $\phi$ is induced by $f$ and $\nu$ on $\arcgraphstratum{k_1}{c}{S_1}$, and therefore on $\B{k}{S_1}$ since $c$ is arbitrary. 
\end{proof}

\begin{lem}
The image of $\phi$ is in $\arcgraphstratum{k_2}{\nu}{S_2}$.
\end{lem}

\begin{proof}
Suppose not. Take $\mu \in \arcgraph{k_1}{S_1} \backslash \B{k_1}{S_1}$ such that $\phi(\mu) \notin \arcgraphstratum{k_2}{\nu}{S_2}$. In other words, there exists an arc $a \in \nu$ such that $a \notin \phi(\mu)$. Take a multiarc $\omu \in \B{k_1}{S_1}$ adjacent to $\mu$, which exists by Remark \ref{sep next to nonsep}. We know that $\phi(\omu) = \nu \cup f(\omu)$ and since $\mu$ and $\omu$ are adjacent and $\phi$ is an embedding, $\phi(\mu)$ and $\phi(\omu)$ are adjacent. 
Therefore, $|\phi(\mu) \cap \phi(\omu)| = k_2 -1$ and $\phi(\mu) \cap \phi(\omu) = \nu \backslash \{a\} \cup f(\omu)$. This tells us that $f(\omu) \subseteq \phi(\mu)\backslash \nu =: \eta$ and $\nu \backslash \{a\} \subset \phi(\mu)$ where, $|\eta| = k_2 - (d-1) = k_1+1$. 
Then $\omu \subseteq f\inv (\eta)$, which has at most cardinality $k_1 + 1$ since $f$ is injective. This holds true for every $\omu \in \B{k_1}{S_1}$ adjacent to $\mu$, and $\eta$ does not depend on $\omu$. We know there are at most $k_1+1$ multiarcs in $f\inv(\eta)$ but there are more than $k_1+1$ arcs in $\B{k_1}{S_1}$ adjacent to $\mu$ by Remark \ref{sep next to nonsep} arriving at a contradiction.
\end{proof}

This gives us a map 
\begin{align*}
\Phi: \arcgraph{k_1}{S_1} &\lra \arcgraph{k_1}{S_2} \\
	\mu &\longmapsto \phi(\mu) \backslash \nu
\end{align*}
So we have
\begin{equation*}
\phi(\mu) = \Phi(\mu) \cup \nu
\end{equation*}
for every $\mu \in \arcgraph{k_1}{S_1}$ and all that's left to show is that $f$ induces $\Phi$.

\begin{lem}
For any arc $b$ on $S_1$, $\Phi(\arcgraphstratum{k_1}{b}{S_1}) \subseteq \arcgraphstratum{k_1}{f(b)}{S_2}$.
\end{lem} 

\begin{proof}
Take $\mu \in \arcgraphstratum{k_1}{b}{S_1}$. If $\mu \in \arcgraphstratum{k_1}{b}{S_1} \cap \B{k_1}{S_1}$, then $\Phi(\mu) = f(\mu) \in \arcgraphstratum{k_2}{f(b)}{S_2}$ by Lemma \ref{On B}. 
Suppose $\mu \in \arcgraphstratum{k_1}{b}{S_1} \backslash \B{k_1}{S_1}$. For any $\omu \in \B{k_1}{S_1}$ that is adjacent to $\mu$ we have
\begin{equation*}
\Phi(\omu) \cap \Phi(\mu) = f(\omu) \backslash \{f(b)\}. 
\end{equation*}
Therefore $\omu$ is a subset of $\Phi(\mu) \cup \{f(b)\}$. This implies that there are $k_1 + 1$ multiarcs in $\B{k_1}{S_1}$ adjacent to $\mu$, contradicting Remark \ref{sep next to nonsep}.
\end{proof}

Take any $\mu \in \arcgraph{k_1}{S_1}$ and say $\mu = \{b_1, \ldots, b_{k_1} \}$. Then 
\begin{equation*}
\Phi(\mu) = \cap_{j=1}^{k_1} \Phi(\arcgraphstratum{k_1}{b_j}{S_1}) \subseteq \cap_{j=1}^{k_1} \arcgraphstratum{k_1}{f(b_j)}{S_2} = f(\mu).
\end{equation*}
Therefore $\Phi(\mu) = f(\mu)$ for all multiarcs $\mu \in \arcgraph{k_1}{S_1}$, which concludes the proof of Theorem \ref{embedding main}.
%%%%%%%%%%%%%%%%%%%%%%%%%%%%%%%%%%%%%%%%%%%%%%%%%%%%%%%%%%%%%%%%%

{\it Addresses:}\\
Department of Mathematics, University of Luxembourg, Esch-sur-Alzette, Luxembourg\\
Department of Mathematics, Brown University, Providence, RI, USA

{\it Emails:}\\
hugo.parlier@uni.lu\\
aweber@math.brown.edu

\end{document}